\documentclass[11pt]{article}
\usepackage[comma,numbers,square,sort&compress]{natbib}
\usepackage{color}
\usepackage{}
\usepackage{amsmath}
\usepackage{amsmath,amssymb,amsthm}
\usepackage[a4paper,left=20mm,right=20mm, top=30mm,body={360mm,240mm}]{geometry}
\usepackage{indentfirst}
\usepackage{CJK}
\begin{document}

\renewcommand{\abstractname}{}
\renewcommand{\thefootnote}{}
\renewcommand{\proofname}{\textbf{Proof}}
\newtheorem{theorem}{Theorem}[section]
\newtheorem{lemma}{Lemma}[section]
\newtheorem{corollary}{Corollary}[section]
\newtheorem{definition}{Definition}[section]
\numberwithin{equation}{section}
\newtheorem{remark}{Remark}[section]
\newtheorem{example}{Example}
\newtheorem{proposition}{Proposition}[section]
\title{  A norm inequality on noncommutative symmetric spaces related to a question of Bourin}
\author{Jinchen Liu\; Kan He\; Xingpeng Zhao \footnote{\small Corresponding author: Xingpeng Zhao; Email: zhaoxingpeng1@sina.com}\\
{Department of Mathematics, Taiyuan University of Technology,}\\ {030024, TaiYuan, Shanxi, China}
}

\date{}
\maketitle
\begin{abstract}
\noindent \textbf{Abstract}~~ In this note, we study a question introduced by Bourin \cite{2009Matrix} and partially solve the question of Bourin.
In fact, for $t\in[0,\frac{1}{4}]\cup[\frac{3}{4},1]$, we show that
$$|||x^{t}y^{1-t}+y^{t}x^{1-t}|||\leq|||x+y|||,$$
where $x,y\in\mathbb{M}_{n}(\mathbb{C})^+$ and $\||\cdot\||$ is the unitarily invariant norm. Moreover, we prove that
the above inequality holds on noncommutative fully symmetric spaces.\\

\noindent\nonumber\textbf{Keywords}~~$\tau$-measurable operators;\ fully symmetric spaces;\ von Neumann algebra

\noindent\nonumber\textbf{Mathematics Subject Classification}~~47A30 $\cdot$
47B15
 $\cdot$ 15A60
\end{abstract}

\section{Introduction}

Throughout this note, we denote the collection of all $n\times n$ complex matrices by $\mathbb{M}_{n}(\mathbb{C})$ and the collection of all positive semidefinite $n\times n$ complex matrices by $\mathbb{M}_{n}(\mathbb{C})^+$. Recall that $\||\cdot\||$ is unitarily invariant on $\mathbb{M}_{n}(\mathbb{C})$, if $|||UAV|||=|||A|||$ for all $A\in{\mathbb{M}_{n}(\mathbb{C})}$ and all unitaries $U,V\in{\mathbb{M}_{n}(\mathbb{C})}$.

For a period of time, Bourin \cite{2009Matrix} was committed to subadditivity of concave functions of positive semidefinite matrices. He posed an interesting conjecture:
\begin{equation}\label{inequality 1}
|||A^pB^q+B^pA^q|||\leq|||A^{p+q}+B^{p+q}|||,
\end{equation}
where $p,q\geq0$ and $A,B\in \mathbb{M}_{n}(\mathbb{C})^{+}$. Obviously, if $A$ and $B$ are replaced by $A^{\frac{1}{p+q}}$ and $B^{\frac{1}{p+q}}$ in Inequality (\ref{inequality 1}) respectively, one can easily check that Inequality (\ref{inequality 1}) is equivalent to the following inequality:
\begin{equation}\label{inequality 2}
\||A^{t}B^{1-t}+B^{t}A^{1-t}\|| \leq\||A+B\||,
\end{equation}
where $t\in[0,1]$.
In the past few decades, more and more authors are interested in Bourin's conjecture. Up to now, this conjecture has not been fully resolved. However, there have been some fruitful outcomes. The readers can refer to \cite{2013TRACE, 2014Trace, 2021Norm, 2022Norm} and the references therein for more details. For example,
 Hayajneh-Hayajneh-Kittaneh \cite{2021Norm} proved
Inequality (\ref{inequality 2}) for $t=\frac{1}{4}$, $\frac{1}{2}$ and $\frac{3}{4}$. Following closely,
Hayajneh-Hayajneh-Kittaneh-Lebaini \cite{2022Norm} made an important improvement on
Inequality (\ref{inequality 2}).
Let $A, B$ be $n\times n$ positive semifinite matrices. For $t\in[0,1]$, it is shown that
\begin{equation}\label{conclusion(6)}
\||A^{t}B^{1-t}+B^{t}A^{1-t}\|| \leq\ 2^{|\frac{1}{2}-t|}||A+B\||
\end{equation}
and
\begin{equation}\label{conclusion(4)}
\||B^{1-t}A^{2t-1}B^{1-t}+A^{1-t}B^{2t-1}A^{1-t}\||\leq2^{2t-1}\||A+B\||, t\in[\frac12, 1].
\end{equation}
Among these works, we are guided by the results due to Hayajneh-Hayajneh-Kittaneh's work \cite{2023A}.
 Specifically, the authors in \cite{2023A}  got an improvement on
Inequality (\ref{conclusion(4)}):
For $t\in[\frac{3}{4},1]$, it is shown that
\begin{equation}\label{conclusion(5)}
\||B^{1-t}A^{2t-1}B^{1-t}+A^{1-t}B^{2t-1}A^{1-t}\||\leq2^{4(t-\frac{3}{4})}\||A+B\||.
\end{equation}
Besides, by using Inequality (\ref{conclusion(4)}), the authors in \cite{2023A} also got an improvement on Inequality (\ref{conclusion(6)}):
\begin{equation}\label{conclusion(7)}
\||A^{t}B^{1-t}+B^{t}A^{1-t}\|| \leq\ 2^{2(t-\frac{3}{4})}||A+B\||, \;\;\;t\in[\frac{3}{4},1],
\end{equation}
and
\begin{equation}\label{conclusion(8)}
\||A^{t}B^{1-t}+B^{t}A^{1-t}\|| \leq\ 2^{2(\frac{1}{4}-t)}||A+B\||, \;\;\;t\in[0,\frac{1}{4}].
\end{equation}

In this note, we show Inequality (\ref{inequality 2}) holds for $t\in[0,\frac{1}{4}]\cup[\frac{3}{4},1]$, which partially resolves the question of Bourin.
 Especially, we will consider Inequality (\ref{inequality 2}) for $t\in[0,\frac{1}{4}]\cup[\frac{3}{4},1]$ on a noncommutative fully symmetric space  $E(\mathcal{M})$. Let $\mathcal{M}$ be a semifinite von Neumann algebra, equipped with a fixed
faithful, normal, semi-finite trace $\tau$, and $E$ is a fully symmetric space defined on $(0,\infty]$. Then the  noncommutative fully symmetric space
is defined as \[E(\mathcal{M}):=\{x\in S(\tau);\mu(x)\in E\} .\]
We equip $E(\mathcal{M})$ with a natural norm
$$||x||_{E(\mathcal{M})}:=||\mu(x)||_{E},x\in E(\mathcal{M}).  $$
 Moreover, by adapting the techniques in \cite{1986Generalized,2020Arithmetic},
we show that: when $0\leq x,y\in  E(\mathcal{M})$,
\begin{equation}\label{conclusion(9)}
||x^{t}y^{1-t}+y^{t}x^{1-t}||_{E(\mathcal{M})}\leq||x+y||_{E(\mathcal{M})}
\end{equation}
 for $t\in[0,\frac{1}{4}]\cup[\frac{3}{4},1]$, where $\mathcal{M}$ is a semifinite von Neumann algebra and  $E$ is a fully symmetric space.

\section{Preliminaries}
In this present paper,
$\mathcal{M}$ will be denoted by a semifinite
von Neumann algebra, equipped with a fixed
faithful, normal, semi-finite trace $\tau$.

We denote the projection
lattice of $\mathcal{M}$ by $P(\mathcal{M})$ and write
$P_f(\mathcal{M}):=\{e\in P(\mathcal{M}): \tau(e)<\infty\}$.
A closed densely defined linear operator $x$ on $\mathcal {H}$ with
domain $D(x)$ is said to be affiliated with $\mathcal {M}$ if and
only if $u^*xu=x$ for all unitary operators $u$ which belong to
the commutant of $\mathcal {M}$. When $x$ is
affiliated with $\mathcal {M}$, $x$ is said to be
$\tau$-measurable if for every $\varepsilon>0$ there exists $e\in
P(\mathcal{M})$ such that $e(\mathcal{H})\subseteq D(x)$ and
$\tau(e^{\perp})<\varepsilon$ (where $e^{\perp }=1-e$). The set of all $\tau$-measurable operators will
be denoted by $L_0(\mathcal{M})$. The set
$L_0(\mathcal{M})$ is a $\ast$-algebra with sum and product
being the respective closure of the algebraic sum and product.
The measure topology in $L_0(\mathcal{M})$ is
the vector space topology defined via the neighbourhood
base $\{V(\varepsilon, \delta): \varepsilon, \delta>0\}$, where
$V(\varepsilon, \delta)=\{x\in L_0(\mathcal{M}):
\tau(e_{(\varepsilon,\infty)}(|x|))\leq \delta\}$ and
$ e_{(\varepsilon,\infty)}(|x|)$ is the spectral projection
of $|x|$ associated with the interval $(\varepsilon, \infty)$.
With respect to the measure topology, $L_0(\mathcal{M})$
is a complete   topological $*$-algebra.

For every $x\in L_0(\mathcal{M})$, there is a unique polar
decomposition $x=u|x|$, where $|x|\in \mathcal{M}^+$
(the positive part of $\mathcal{M}$) and $u$ is a
partial isometry operator.
Let $r(x)=u^*u$ and $l(x)=uu^*$.  We call
$r(x)$ and $l(x)$ the right and left supports of $x$, respectively.
If $x$ is self-adjoint, then $r(x)=l(x)$.
This common projection is
then said to be the support of $x$ and denoted by $s(x)$.

For $x\in L_0(\mathcal{M})$, we define
$$
\lambda_t(x)=\tau(e_{(t, \infty)}(|x|))~~~~~~~~ \mbox{and}~~~~~~~
\mu_t(x)=\inf \{s>0: \lambda_s(x)\leq t\}.
$$
 The
function $t\rightarrow \lambda_t(x)$ is called the distribution function of $x$ and $t\rightarrow
\mu_t(x)$ is the generalized singular number
of $x$. We will denote simply
by $\lambda(x)$ and $\mu(x)$ the functions
$t\rightarrow \lambda_t(x)$  and $t\rightarrow
\mu_t(x)$, respectively(cf. \cite{1986Generalized}).

If $x,y\in L_{0}(\mathcal{M})$, then $x$ is said to be submajorized  $y$, denoted by  $x\prec \prec y$,  if and only if
$$\int_{0}^{a}\mu_{t}(x)dt\leq\int_{0}^{a}\mu_{t}(y)dt,\forall a\geq0.$$

 To achieve our main results, we state some known properties of $\mu(\cdot)$ without proof.
\begin{lemma}\label{lemma 2.1}(\cite{1986Generalized})
Let $x, y\in  L_0(\mathcal{M})$. Then
\begin{enumerate}
\item $\mu_\cdot(|x|)=\mu_\cdot(x)=\mu_\cdot(x^*)$ and $\mu_\cdot(\alpha x)=|\alpha|\mu_\cdot(x)$ for  $\alpha\in \mathbb{C}$;
\item let $f$ be a bounded continuous
 increasing function on $[0, \infty)$ with $f(0)=0$;
Then $\mu_\cdot(f(x))=f(\mu_\cdot(x))$ and $\tau(f(x))=\int_0^{\tau(1)} f(\mu_t(x))dt$;
\item if $0\leq x\leq y$, then $\mu_\cdot(x)\leq\mu_\cdot(y)$;
\item if $a, b\in \mathcal{M}$, then $\mu_\cdot(axb)\leq\|a\|\|b\|\mu_\cdot(x)$;
 \item $\int_0^af(\mu_{s}(x+y))ds\leq\int_0^af(\mu_s(x)+\mu_s(y))ds$, for $a>0$ and any convex increasing function
 $f: \mathbb{R}^+\rightarrow \mathbb{R}$;

 \item $\int_0^af(\mu_{s}(xy))ds\leq\int_0^af(\mu_s(x)\mu_s(y))ds$, for $a>0$ and any increasing function
 $f: \mathbb{R}^+\rightarrow \mathbb{R}$ such that
 $t\rightarrow f(e^t)$ is convex.
\end{enumerate}
\end{lemma}
Denote by $S(\tau)$ the collection of all $\tau-$measurable operators. Let $E$ be a linear subspace of $S(\tau)$ equipped with a complete
norm $||\cdot||_{E}.$ Let $(I,m)$ denote the measure space, where $I=(0,a)$ and $a\in(0,\infty)$, equipped with Lebesgue measure $m$. If $\forall f\in E$, $ g\in L_{0}(I,m)$ with $g\prec\prec f$, then $||g||_{E}\leq||f||_{E}$ and $g\in E$. We say that $E$ is a fully symmetric space.

Let $E$ be a fully symmetric Banach function space on $(0,\tau(1))$. Define
\[E(\mathcal{M}):=\{x\in S(\tau);\mu(x)\in E\} \]
and equip $E(\mathcal{M})$ with a natural norm
$$||x||_{E(\mathcal{M})}:=||\mu(x)||_{E},x\in E(\mathcal{M}).  $$
Then $E(\mathcal{M})$ is a Banach space with the norm $||\cdot||_{E(\mathcal{M})}$ and is called the noncommutative
 fully symmetric operator space associated with $\mathcal{M}$ corresponding to $(E,||\cdot||_{E})$ (see \cite{2008Symmetric, 2012Singular}).
 In the following discussion, $E$ and $\mathcal{M}$ stand for a fully symmetric space and a semifinite
von Neumann algebra respectively defined on $(0,\tau(1))$ unless otherwise noted.

\section{Main results}\label{section3}

To achieve our main results, we start with some lemmas which will be used in later proof.

Let $0\leq x, y\in L_0(\mathcal{M})$ and $z\in\mathcal{M}$. For convenience, set
$$b_t=x^ty^{1-t}+y^{t}x^{1-t},~~~B_t=x^{t}zy^{1-t}+y^tz^*x^{1-t}, t\in[0,1]$$
and if $t\in[\frac{1}{2},1]$, put
$$f_t=y^{1-t}x^{2t-1}y^{1-t}+x^{1-t}y^{2t-1}x^{1-t}$$
and
$$F_t=y^{1-t}z^{*}x^{2t-1}zy^{1-t}+x^{1-t}zy^{2t-1}z^*x^{1-t}.$$
Then we have a result described by the following lemma. It should be noted that its proof is exactly the same to matrix case \cite[Lemma 2.1]{2021Norm} by using \cite[Lemma 3.32]{2021Connections}.
\begin{lemma}\label{lemma3.1}
Let $0\leq x, y\in L_{0}(\mathcal{M})$, and let $z\in \mathcal{M} $. If  $ t\in[\frac{1}{2},1]$, then
$$\begin{bmatrix} x+y & B_t\\ B_{t}^* & F_t\\
\end{bmatrix} \geq 0.$$
Moreover,  $B_{t}=(x+y)^{\frac{1}{2}}c(F_t)^{\frac{1}{2}}$  for some contraction $c\in\mathcal{M}$.
\end{lemma}
\begin{proof} An elementary calculation shows that
\begin{align}
\begin{bmatrix} x+y & B_t\\ B_{t}^* & F_t\\ \end{bmatrix}
&=\begin{bmatrix} x^{\frac{1}{2}} & y^{\frac{1}{2}}\\
y^{1-t}z^*x^{\frac{2t-1}{2}} & x^{1-t}zy^{\frac{2t-1}{2}}\\
\end{bmatrix}
\begin{bmatrix} x^{\frac{1}{2}} & y^{\frac{1}{2}}\\ y^{1-t}z^*x^{\frac{2t-1}{2}} & x^{1-t}zy^{\frac{2t-1}{2}}\\ \end{bmatrix}^{*}\geq{0}.\nonumber
\end{align}
By applying \cite[Lemma 3.32]{2021Connections}, we can get $B_{t}=(x+y)^{\frac{1}{2}}c(F_t)^{\frac{1}{2}}$  for some contraction $c\in\mathcal{M}$.
\end{proof}

The following lemma can be found in \cite{2019Logarithmic} and \cite{2020Arithmetic}.
\begin{lemma}\label{lemma 3.2}
Let $x, y\in  L_0(\mathcal{M})$. Then
\begin{enumerate}

\item (\cite[Proposition 5.6]{2019Logarithmic}) let $xy\in E(\mathcal{M})$. If $m>1$, $\frac{1}{m}+\frac{1}{n}=1$, then $$||xy||_{E(\mathcal{M})}\leq|| |x|^m ||_{E(\mathcal{M})}^{\frac{1}{m}} || |y|^n ||_{E(\mathcal{M})}^{\frac{1}{n}};$$
\item (\cite[Corollary 5.8]{2020Arithmetic}) let $z\in L_{0}(\mathcal{M})$ with $x, y\geq0$ and $t\in[0,1]$.
If $xz, zy\in E(\mathcal{M})$, then
$$2||x^\frac{1}{2}zy^\frac{1}{2}||_{E(\mathcal{M})}\leq||x^{t}zy^{1-t}+x^{1-t}zy^{t}||_{E(\mathcal{M})}\leq||xz+zy||_{E(\mathcal{M})}$$
and
$$||x^{t}zy^{1-t}||_{E(\mathcal{M})}\leq||xz||_{E(\mathcal{M})}^{t}||zy||_{E(\mathcal{M})}^{1-t}.$$

\end{enumerate}
\end{lemma}

In the following, we denote by $\mathbb{M}_2(\mathcal{M})$ the semifinite von Neumann algebra
\[
\mathbb{M}_2(\mathcal{M}):=\left\{ \begin{bmatrix} x_{11} & x_{12}\\ x_{21} & x_{22}\\ \end{bmatrix}: x_{ij}\in\mathcal{M}, i,j=1,2\right\}
\]
on a Hilbert space $\mathcal{H}\oplus \mathcal{H}$ with the trace $tr\otimes\tau.$

\begin{lemma}\label{lemma3.6}
Let $0\leq x,y\in L_{0}(\mathcal{M})$. Then
$$\mu_\cdot(x\oplus y)\prec\prec\mu_\cdot(x+y). $$
\end{lemma}
\begin{proof}
One can easily check that  $\mu_\cdot (x+y)=\mu_\cdot((x+y)\oplus0)$. It is sufficient to show that $\mu_\cdot(x\oplus y)\prec\prec\mu_\cdot((x+y)\oplus0)$.
Let $T=\begin{bmatrix} x^{\frac{1}{2}} & y^{\frac{1}{2}} \\ 0 & 0 \\ \end{bmatrix}$ and $T^*=\begin{bmatrix} x^{\frac{1}{2}} & 0 \\ y^{\frac{1}{2}} & 0 \\ \end{bmatrix}$.
It follows that
\begin{align}
\mu_\cdot((x+y)\oplus0)&=\mu_\cdot(\begin{bmatrix} x^{\frac{1}{2}} & y^{\frac{1}{2}} \\ 0 & 0 \\ \end{bmatrix}\begin{bmatrix} x^{\frac{1}{2}} & 0 \\ y^{\frac{1}{2}} & 0 \\ \end{bmatrix})\nonumber\\
&=\mu_\cdot(\begin{bmatrix} x^{\frac{1}{2}} & 0 \\ y^{\frac{1}{2}} & 0 \\ \end{bmatrix}\begin{bmatrix} x^{\frac{1}{2}} & y^{\frac{1}{2}} \\ 0 & 0 \\ \end{bmatrix})\nonumber\\
&=\mu_\cdot(\begin{bmatrix} x & x^{\frac{1}{2}}y^{\frac{1}{2}} \\ y^{\frac{1}{2}} x^{\frac{1}{2}} & y \\ \end{bmatrix}).\nonumber
\end{align}
By Lemma\ \ref{lemma 2.1}(5), we have
\begin{align}
\int_{0}^{a}\mu_{s}(x\oplus y)ds&=\int_{0}^{a}\frac{1}{2}\mu_{s}(\begin{bmatrix} x & x^{\frac{1}{2}}y^{\frac{1}{2}} \\ y^{\frac{1}{2}} x^{\frac{1}{2}} & y \\ \end{bmatrix}+\begin{bmatrix} 1 & 0 \\ 0 & -1 \\ \end{bmatrix}\begin{bmatrix} x & x^{\frac{1}{2}}y^{\frac{1}{2}} \\ y^{\frac{1}{2}} x^{\frac{1}{2}} & y \\ \end{bmatrix}\begin{bmatrix} 1 & 0 \\ 0 & -1 \\ \end{bmatrix})ds\nonumber\\
&\leq\int_{0}^{a}\mu_{s}(\begin{bmatrix} x & x^{\frac{1}{2}}y^{\frac{1}{2}} \\ y^{\frac{1}{2}} x^{\frac{1}{2}} & y \\ \end{bmatrix})ds,\nonumber
\end{align}
which implies that 
$$\mu_\cdot(x\oplus y)\prec\prec \mu_\cdot(x+y). $$
\end{proof}

Next, we will introduce an important result which follows immediately from \cite[Page 4, (9)]{2021On} and \cite[Proposition 3.2]{2019Logarithmic}.

\begin{lemma}\label{lemma3.8}
Let $x, y\in  L_0(\mathcal{M})$. Then
\begin{enumerate}
\item  let $xy\in L_{0}(\mathcal{M}), p>0$. If $xy$ is self-adjoint, then
 $$\int_0^t\mu_s(xy)^pds \leq \int_0^t\mu_s(yx)^pds;$$
\item (\cite[Proposition 2.4]{2016On}) let $x,y \geq 0$. If $r\geq 1$, then
$$|xy|^{r}\prec\prec x^{r}y^{r}.$$
\end{enumerate}
\end{lemma}

Now, we can get an inequality on $B_t$ and $F_t$, described in the following theorem.

\begin{theorem}\label{theorem3.1}
Let $0\leq x,y\in E(\mathcal{M})$ and $z\in\mathcal{M}$ for $t\in[\frac{1}{2},1]$. Then
$$ ||B_t||_{E(\mathcal{M})}\leq||x+y||_{E(\mathcal{M})}^{\frac{1}{2}}||F_t||_{E(\mathcal{M})}^{\frac{1}{2}}.$$
\end{theorem}
\begin{proof}
From Lemma\ \ref{lemma3.1} and Lemma\ \ref{lemma 3.2}(2), we can get
\begin{align*}
||B(t)||_{E(\mathcal{M})}&=||(x+y)^{\frac{1}{2}}cF_t^{\frac{1}{2}}||_{E(\mathcal{M})}\nonumber\\
&\leq||(x+y)c||^{\frac{1}{2}}_{E(\mathcal{M})}||cF_{t}||^{\frac{1}{2}}_{E(\mathcal{M})}\nonumber\\
&\leq||(x+y)||^{\frac{1}{2}}_{E(\mathcal{M})}||F_{t}||^{\frac{1}{2}}_{E(\mathcal{M})}.\nonumber
\end{align*}
It follows from H\"{o}lder inequality that  $||B_t||_{E(\mathcal{M})}\leq||x+y||_{E(\mathcal{M})}^{\frac{1}{2}}||F_t||_{E(\mathcal{M})}^{\frac{1}{2}}.$
\end{proof}

 In particular, take $z=1$ in Theorem  \ref{theorem3.1} and then we can easily obtain the following result.
\begin{corollary}\label{corollary3.1}
Let $0\leq x,y\in E(\mathcal{M})$. If  $t\in[\frac{1}{2},1]$, then
$$ ||b_t||_{E(\mathcal{M})}\leq||x+y||^{\frac{1}{2}}_{E(\mathcal{M})}||f_t||^{\frac{1}{2}}_{E(\mathcal{M})}.$$
\end{corollary}

Now let's introduce a main result in this note as follows.
\begin{proposition}\label{theorem3.4}
Let $0\leq x,y\in E(\mathcal{M})$. Then

$$||f_{\frac{3}{4}}||_{E(\mathcal{M})}\leq||x+y||_{E(\mathcal{M})}.$$

\end{proposition}
\begin{proof}
For $t=\frac{3}{4}$, it follows from Lemma\;\ref{lemma 2.1} and Lemma\;\ref{lemma3.8}
 that
\begin{align}
\int_0^a\mu_s(f_{\frac{3}{4}})ds&\leq\int_0^a\mu_s(y^{\frac{1}{4}}x^{\frac{1}{2}}y^{\frac{1}{4}})
+\mu_s(x^{\frac{1}{4}}y^{\frac{1}{2}}x^{\frac{1}{4}})ds\nonumber \\
&\leq\int_0^a\mu_s(x^{\frac{1}{2}}y^{\frac{1}{2}})
+\mu_s(x^{\frac{1}{2}}y^{\frac{1}{2}})ds\nonumber\nonumber\\
&=\int_0^a2\mu_s(x^{\frac{1}{2}}y^{\frac{1}{2}})ds \nonumber.
\end{align}
Combining this with Lemma  \ref{lemma 3.2}, we have
$||f_{\frac{3}{4}}||_{E(\mathcal{M})}\leq2\|x^{\frac{1}{2}}y^{\frac{1}{2}}\|_{E(\mathcal{M})}\leq||x+y||_{E(\mathcal{M})}.$
\end{proof}

Based on the facts above, wee will get a partial improvement on Inequality (\ref{conclusion(7)}) stated as follows.
\begin{theorem}\label{theorem3.2}
Let $0\leq x,y\in E(\mathcal{M})$. Then for $t\in[\frac{3}{4},1]$,
$$||f_{t}||_{E(\mathcal{M})} \leq ||x+y||_{E(\mathcal{M})}.$$
\end{theorem}
\begin{proof}
The case $t=1$ is obvious and the result for $t=\frac{3}{4}$ has been proved in Proposition\;\ref{theorem3.4}. Thus we only need to show the case $\frac{3}{4}<t<1$. For any $t\in(\frac{3}{4},1)$, set $r=2t-1$. Denote
$$X=
\begin{bmatrix}
y^{\frac{1-r}{2}}x^{\frac{1-r}{2}} & x^{\frac{1-r}{2}}y^{\frac{1-r}{2}}\\
0 & 0\\
\end{bmatrix}$$
and
$$D=
\begin{bmatrix}
x^{2r-1} & 0\\
0 & y^{2r-1}\\
\end{bmatrix}.$$
Let $m=\frac{1}{2(1-r)}$ and $n=\frac{1}{2r-1}$.  Since $t\in(\frac{3}{4},1)$, it follows that $m, n>1$ and $m\downarrow1$ as $t\downarrow \frac34$.
Noting that $\frac{1}{m}+\frac{1}{n}=2(1-r)+2r-1=1$, we have from Lemma \ref{lemma3.8}(1) that

\begin{align*}\int_{0}^{a}\mu_{s}(f_t)ds
&=\int_{0}^{a}\mu_{s}(f_t\oplus0)ds \nonumber\\
&=\int_{0}^{a}\mu_{s}
(\begin{bmatrix}
y^{\frac{1-r}{2}}x^{\frac{1-r}{2}} & x^{\frac{1-r}{2}}y^{\frac{1-r}{2}}\\
0 & 0\\
\end{bmatrix}
\begin{bmatrix}
x^{2r-1} & 0\\
0 & y^{2r-1}\\
\end{bmatrix}
\begin{bmatrix}
 x^{\frac{1-r}{2}}y^{\frac{1-r}{2}} & 0\\
 y^{\frac{1-r}{2}}x^{\frac{1-r}{2}} & 0\\
\end{bmatrix})ds \nonumber\\
&= \int_{0}^{a}\mu_{s}(XDX^{*})ds \nonumber \\
&\leq\int_{0}^{a}\mu_{s}(X^*XD)ds.\nonumber
\end{align*}
It follows that
$$||f_t||_{E(\mathcal{M})}\leq||X^*XD||_{E(\mathcal{M})}.$$
By Lemma\ \ref{lemma 3.2}(1), we obtain
\begin{align*}
||X^*XD||_{E(\mathcal{M})}&\leq||D||_{E(\mathcal{M})^{(n)}}||X^{*}X||_{E(\mathcal{M})^{(m)}}\nonumber\\
&=|||D|^{n}||^{\frac{1}{n}}_{E(\mathcal{M})}|||X^{*}X|^{m}||^{\frac{1}{m}}_{E(\mathcal{M})}.\nonumber
\end{align*}
Combining the inequality above with \;Lemma\ \ref{lemma3.6}, we have that
\begin{align*}
|||D|^{n}||^{\frac{1}{n}}_{E(\mathcal{M})}&=||\mu_{s}(D)^{n}||_{E}^{\frac{1}{n}}\nonumber\\
&=||\mu_{s}(x^{2r-1}\oplus y^{2r-1})^{n}||_{E}^{\frac{1}{n}}\nonumber\\
&=||\mu_{s}(x\oplus y)||_{E}^{\frac{1}{n}}\nonumber\\
&\leq||\mu_{s}(x+ y)||_{E}^{\frac{1}{n}}\nonumber\\
&=||x+y||_{E(\mathcal{M})}^{\frac{1}{n}}.\nonumber
\end{align*}
In view of Lemma\;\ref{lemma 2.1} and Lemma\;\ref{lemma3.8}(2),
we obtain
\begin{align*}
\int_{0}^{a}\mu_{s}(X^*X)^{m}ds &\leq \int_{0}^{a}\mu_{s}(y^{\frac{1-r}{2}}x^{1-r}y^{\frac{1-r}{2}}+x^{\frac{1-r}{2}}y^{1-r}x^{\frac{1-r}{2}})^{m}ds\nonumber \\
&\leq \int_{0}^{a}[2\mu_{s}(x^{\frac{1-r}{2}}y^{\frac{1-r}{2}})^{2}]^{m}ds)\nonumber \\
&=\int_{0}^{a}2^{m}\mu_{s}(x^{\frac{1-r}{2}}y^{\frac{1-r}{2}})^{2m}ds)\nonumber \\
&\leq\int_{0}^{a}2^{m}\mu_{s}(x^{\frac{1}{2}}y^{\frac{1}{2}})ds)\nonumber \\
&=\int_{0}^{a}\mu_{s}(2^{m}x^{\frac{1}{2}}y^{\frac{1}{2}})ds),\nonumber
\end{align*}
which shows that
$$|X^{*}X|^{m}\prec\prec2^{m}x^{\frac{1}{2}}y^{\frac{1}{2}}.$$
Then we can conclude from Lemma\ \ref{lemma 3.2}(2) that
\begin{align*}
|||X^{*}X|^{m}||^{\frac{1}{m}}_{E(\mathcal{M})}&\leq||2^{m}x^{\frac{1}{2}}y^{\frac{1}{2}}||^{\frac{1}{m}}_{E(\mathcal{M})}\nonumber\\
&\leq2||x^{\frac{1}{2}}y^{\frac{1}{2}}||^{\frac{1}{m}}_{E(\mathcal{M})}\nonumber\\
&\leq 2^{1-\frac1m}||x+y||^{\frac{1}{m}}_{E(\mathcal{M})}.\nonumber
\end{align*}
Therefore,
\begin{align*}
||f_t||_{E(\mathcal{M})}&\leq||X^*XD||_{E(\mathcal{M})}\nonumber\\
&\leq||D||_{E(\mathcal{M})^{(n)}}||X^{*}X||_{E(\mathcal{M})^{(m)}}\nonumber\\
&=|||D|^{n}||^{\frac{1}{n}}_{E(\mathcal{M})}|||X^{*}X|^{m}||^{\frac{1}{m}}_{E(\mathcal{M})}\nonumber\\
&\leq2^{1-\frac1m}||x+y||^{\frac{1}{m}}_{E(\mathcal{M})}||x+y||^{\frac{1}{n}}_{E(\mathcal{M})}\nonumber\\
&=2^{1-\frac1m}||x+y||_{E(\mathcal{M})}, \mbox{for~all}~m\in (1, \infty). \nonumber
\end{align*}
By the arbitrariness of $m$, it is known that $||f_t||_{E(\mathcal{M})}\leq||x+y||_{E(\mathcal{M})}$.
\end{proof}
Applying the method of proving Theorem\;\ref{theorem3.2}, we can directly obtain the following result without proof.

\begin{proposition}
Let $0\leq x,y\in \mathbb{M}_{n}(\mathbb{C})$. Then for $t\in[\frac{3}{4},1]$,

$$\||f_{t}\|| \leq\||x+y\||,$$
where $\||\cdot\||$ is the unitarily invariant norm.

\end{proposition}

Using Corollary\;\ref{corollary3.1} and Theorem\;\ref{theorem3.2}, we can get the following theorem.
\begin{theorem}\label{theorem3.3}
Let $0\leq x,y\in  E(\mathcal{M})$. If $t\in[0,\frac{1}{4}]\cup[\frac{3}{4},1]$, then
$$||b_{t}||_{E(\mathcal{M})}\leq||x+y||_{E(\mathcal{M})}.$$
\end{theorem}
\begin{proof}
When $t\in[\frac{3}{4},1]$, we have from Corollary\;\ref{corollary3.1} and Theorem\;\ref{theorem3.2} that
$$||b_{t}||_{E(\mathcal{M})}\leq||x+y||_{E(\mathcal{M})}.$$
When $t\in[0,\frac{1}{4}]$, it follows that 
\begin{align*}
||b_{t}||_{E(\mathcal{M})}&=||b_{t}^{*}||_{E(\mathcal{M})}\nonumber \\
&=||b_{1-t}||_{E(\mathcal{M})} \; (b_{t}^{*}=b_{1-t})\\
&\leq||x+y||_{E(\mathcal{M})}.
\end{align*}
\end{proof}
Similarly to the proof of Theorem\;\ref{theorem3.3}, we can obtain the following result.

\begin{proposition}
Let $0\leq x,y\in  \mathbb{M}_{n}(\mathbb{C})$. If $t\in[0,\frac{1}{4}]\cup[\frac{3}{4},1]$, then

$$\||x^{t}y^{1-t}+y^{t}x^{1-t}\|| \leq\||x+y\||,$$ where $\||\cdot\||$ is the unitarily invariant norm.
That is to say, Inequality (\ref{inequality 2}) holds for $t\in[0,\frac{1}{4}]\cup[\frac{3}{4},1]$.

\end{proposition}

\section*{Funding}
\noindent The article is supported by Fundamental Research Program of Shanxi Province (No.202103021224104) and Fundamental Research Program of Shanxi Province (No.202103021223038) and the Scientific and Technological Innovation Programs of Higher Education Programs in Shanxi (No.2021L015).\\

\section*{ Competing interests}

\noindent The authors declare that there is no conflict of interests regarding the publication of this article.

\section*{ Authors's contributions }

\noindent The authors contributed equally and significantly in writing this article.
All authors read and approved the final manuscript.

\newpage

\bibliographystyle{unsrt}
\bibliography{reference}

\end{document}